\documentclass[12pt]{amsart}
\usepackage{amssymb}
\usepackage[latin2]{inputenc}
\usepackage[english]{babel}
\usepackage[bottom=3cm,top=2cm]{geometry}
\usepackage{mdwlist}
\usepackage{enumerate}

\makeatletter
\@namedef{subjclassname@2010}{%
  \textup{2010} Mathematics Subject Classification}
\makeatother

\newtheorem{theorem}{Theorem}[section]

\newtheorem{proposition}[theorem]{Proposition}
\newtheorem{corollary}[theorem]{Corollary}
\newtheorem{lemma}[theorem]{Lemma}
\newtheorem{claim}[theorem]{Claim}

\theoremstyle{definition}
\newtheorem{definition}[theorem]{Definition}
\newtheorem{remark}[theorem]{Remark}
\newtheorem{example}[theorem]{Example}

\newcommand\comp{\circ}
\newcommand{\conc}{ {}^\frown}
\newcommand{\restr}{\ensuremath\hspace{-3pt}\restriction\hspace{-3pt}}
\newcommand{\id}[1]{\ensuremath{\mathrm{id}_{#1}}}
\newcommand{\inj}[1]{\ensuremath{\mathit{Inj}(#1)}}
\newcommand{\sym}[1]{\ensuremath{\mathit{Sym}(#1)}}
\newcommand{\dom}{\ensuremath{\mathrm{dom}\;}}
\newcommand{\pwrset}[1]{\ensuremath{\mathcal{P}({#1})}}
\renewcommand{\bar}{\overline}

\newcommand{\Fsigma}{\ensuremath{\mathbf{\Sigma}^0_2}}
\newcommand{\Gdelta}{\ensuremath{\mathbf{\Pi}^0_2}}
\newcommand{\analytic}{\ensuremath{\mathbf{\Sigma}^1_1}}
\newcommand{\anform}[1]{\ensuremath{\mathbf{\Sigma}^1_1}({#1})}
\newcommand{\Iminus}[1]{\ensuremath{\mathrm{I}^-(#1)}}
\newcommand{\I}[1]{\ensuremath{\mathrm{I}(#1)}}
\newcommand{\producttop}{\ensuremath{\tau_{\mathrm{pr}}}}

\newcommand{\struktura}[1]{\ensuremath{\mathcal{#1}}}

\newcommand{\At}{\ensuremath{\mathrm{At}}}
\newcommand{\LL}[2]{\ensuremath{{L}_{{#1}{#2}}}}
\newcommand{\Mod}[2]{\ensuremath{\mathrm{Mod}_{#1}^{#2}}}
\DeclareMathAlphabet{\mathpzc}{OT1}{pzc}{m}{it}
\newcommand{\Modf}[1]{\ensuremath{\mathpzc{Mod}_{#1}}}
\newcommand{\Modt}[2]{\ensuremath{\mathpzc{Mod}^{#1}_{#2}}}
\renewcommand{\mod}[3]{\ensuremath{\mathrm{Mod}_{#1}(#2,#3)}}

\begin{document}


\baselineskip=17pt



\title[A dichotomy for the generalized Baire space]{A dichotomy theorem for the generalized Baire space and elementary embeddability at uncountable cardinals}
\author[D. Szir\'aki]{Dorottya Szir\'aki}
  \address{Alfr\'ed R\'enyi 
 Institute of Mathematics\\ 
 Hungarian Academy of Sciences\\
 Re\'altanoda~u.~13--15, H--1053 Budapest, Hungary,
 and
    Department of Mathematics and its Applications\\ Central European University\\
    N\'ador u. 9, H--1051 Budapest, Hungary 
}
  \email{sziraki\_dorottya@phd.ceu.edu}
\author[J. V\"a\"an\"anen]{Jouko V\"a\"an\"anen}
\address{Department of Mathematics and Statistics\\
Gustaf H\"allstr\"omin katu~2b, FI--00014 University of Helsinki, Finland,
 and Institute for Logic, Language and Computation\\ University of Amsterdam\\ 1090 GE Amsterdam, Netherlands}
\email{jouko.vaananen@helsinki.fi}

\date{}

\begin{abstract}
We consider the following dichotomy for $\Fsigma$ finitary relations $R$ on analytic subsets of the generalized Baire space for $\kappa$: either all $R$-independent sets are of size at most $\kappa$, or there is a $\kappa$-perfect $R$-independent set. 
This dichotomy is the uncountable version of a result
found in (W. Kubi\'s, Proc. Amer. Math. Soc.~131~(2003), no~2.:619--623) and in (S. Shelah, Fund. Math.~159~(1999), no.~1:1--50). We prove that the above statement holds assuming $\Diamond_\kappa$ and the set theoretical hypothesis $\Iminus\kappa$, which is the modification of the hypothesis $\I\kappa$ suitable for limit cardinals. When $\kappa$ is inaccessible, or when $R$ is a closed binary relation, the assumption $\Diamond_\kappa$ is not needed.
  
We obtain as a corollary the uncountable version of a  result by G. S\'agi and the first author (Log. J. IGPL~20~(2012), no.~6:1064--1082) about the $\kappa$-sized models of a $\anform{L_{\kappa^+\kappa}}$-sentence when considered up to isomorphism, or elementary embeddability, by elements of a $K_\kappa$ subset of ${}^\kappa\kappa$. 
The role of elementary embeddings can be replaced by a more general notion that also includes embeddings, as well as the maps preserving $L_{\lambda\mu}$ for $\omega\leq\mu\leq\lambda\leq\kappa$ and the finite variable fragments of these logics.
\end{abstract}



\maketitle

\section{Introduction} 
The domain of the \emph{generalized Baire space for $\kappa$}, or the $\kappa$-Baire space for short, is the set ${}^\kappa\kappa$ of functions from $\kappa$ to $\kappa$, and its topology is given by the basic open sets
\[
  N_p=\{x\in{}^\kappa\kappa : p\subseteq x\}
\]
for all $p\in{}^{<\kappa}\kappa$. 
The \emph{generalized Cantor space for} $\kappa$ is ${ }^\kappa 2$ with the topology induced by the $\kappa$-Baire space via the natural injection of ${ }^\kappa 2$ into ${ }^\kappa \kappa$. 
Unless otherwise stated, ${}^\kappa\kappa$ and ${}^\kappa 2$ are endowed with the above topologies, and for $2\leq n<\omega$ the set ${}^{n}({}^{\kappa}\kappa)$ is endowed with the resulting product topology; as in classical descriptive set theory, this space is homeomorphic to ${}^\kappa\kappa$. 
The hypothesis $\kappa^{<\kappa}=\kappa$ is usually assumed when working with the $\kappa$-Baire and $\kappa$-Cantor space, because it
implies that these spaces have some nice properties. Under it, for example, the standard bases of both spaces are of size $\kappa$ and consist of clopen sets, 
and, furthermore, 
the intersection (resp. union) of $<\kappa$ many open (closed) sets is open (closed). 

A subset of a topological space $X$ is defined to be $\Fsigma(\kappa)$ (resp. $\Gdelta(\kappa)$), if it can be written as the union (intersection) of at most $\kappa$ many closed (open) subsets of $X$.
More generally, the collection of \emph{$\kappa$-Borel} subsets of a topological space $X$ is the smallest set of subsets of $X$ which contains the open subsets and is closed under complementation and taking unions and intersections of at most $\kappa$ many sets.
A subset of $X$ is \emph{$\kappa$-analytic}, or $\analytic(\kappa)$, if it can be obtained as a continuous image of a closed subset 
of the $\kappa$-Baire space ${}^\kappa\kappa$;
this concept was introduced in \cite{MeklerVaananen} 
 (for the $\kappa$-Baire space). 
 In
 the case of the $\kappa$-Baire space ${}^\kappa\kappa$ (when $\kappa^{<\kappa}=\kappa$ holds),
 a subset is $\kappa$-analytic iff
 it is the image of a $\kappa$-Borel subset of ${}^\kappa\kappa$ under a $\kappa$-Borel map.
 When the topological space $X$ is homeomorphic to a $\kappa$-analytic subset of the $\kappa$-Baire space, we will omit the ``$\kappa$'' when talking about $\Fsigma(\kappa)$, $\Gdelta(\kappa)$, $\kappa$-Borel or $\kappa$-analytic subsets of $X$, i.e., we will call such subsets $\Fsigma$, $\Gdelta$, Borel and analytic (or $\analytic$) subsets, respectively. 
 
We denote by $\Iminus\kappa$ the following set theoretical hypothesis:
\begin{quotation}
  \noindent 
there is a $\kappa^+$-complete normal (non principal) 
ideal $\mathcal I$ on $\kappa^+$ such that the set 
\[
\mathcal{I}^+=\pwrset{\kappa^+}-\mathcal I
\]
contains a dense subset $K$ such that 
 every descending sequence of length $<\kappa$ of elements of $K$ has a lower bound~in~$K$.
\end{quotation}
This hypothesis is the modification of the hypothesis \I{\kappa} (introduced in~\cite{HodgesShelah}) which is appropriate for limit cardinals $\kappa$ 
 (see also \cite{GalvinJechMagidor, MekSheVaa, SheTuuVaa} and \cite{VaananenCantorBendixson} where the specific case of \I{\omega} is considered).

If $\kappa$ is a regular cardinal and $\lambda>\kappa$ is measurable, then L\'evy-collapsing $\lambda$ to $\kappa^+$ yields a model of ZFC in which $\Iminus\kappa$ holds.
The corresponding statement for $\I\kappa$ is an unpublished result of Richard Laver; the proofs can be reconstructed from the $\I\omega$ case, which is described in \cite{GalvinJechMagidor}. 
Note that $\Diamond_\kappa$ will also hold in the model obtained in the above way, and therefore the consistency of $\Iminus\kappa$ together with $\Diamond_\kappa$ follows from the existence of a measurable $\lambda>\kappa$. 
A model in which $\Iminus\kappa$ holds and $\kappa$ is inaccessible can be obtained in the same way, if we start out from a situation where $\kappa$ is inaccessible and there exists a measurable $\lambda>\kappa$.
 The following two facts were suggested to us by Menachem Magidor.
 \begin{enumerate}[1.]
  \item Given any cardinal $\kappa$, and assuming the existence of a measurable $\lambda>\kappa$, one can also obtain a model in which $2^\kappa>\kappa^+$ holds together with $\Iminus\kappa$ (and also $\Diamond_\kappa$), by forcing with the product of the L\'evy-collapse of $\lambda$ to $\kappa^+$ and $\textit{Add}(\kappa,\mu)$ for any $\mu>\lambda$ 
    (the standard forcing notion that adds $\mu$-many Cohen subsets of $\kappa$).
  \item
It is also possible for a supercompact $\kappa$ to satisfy $\Iminus\kappa$ and even $2^\kappa>\kappa^+$ (see Remark~\ref{weakly compact2}). 
\end{enumerate}
\indent We note that the hypothesis $\kappa^{<\kappa}=\kappa$, which is useful when considering the $\kappa$-Baire space,
 follows from our assumption $\Iminus\kappa$. 
 This 
 implication can be proven by a straightforward generalization of the proof of the fact that $\I\omega$ implies CH \cite[p.~13]{SheTuuVaa}. 
 Using this fact, and a result of Shelah \cite{ShelahDiamonds} that $\Diamond_\kappa$ holds when $\kappa=\mu^+=2^\mu>\aleph_1$, we have that $\Iminus\kappa$ implies $\Diamond_\kappa$ whenever $\kappa>\aleph_1$ is a successor cardinal.
\\[5 pt]
\indent
A subset $Y$ of the generalized Baire space ${}^\kappa\kappa$ is defined to be 
\emph{perfect} if $Y$ consists of the $\kappa$-branches $[T]$ of a \emph{perfect tree} $T$, i.e., a subtree $T\subseteq {}^{<\kappa}\kappa$ whose set of splitting nodes is cofinal and which is $<\kappa$-closed (i.e., every increasing sequence in $T$ of length $<\kappa$ has an upper bound in $T$); 
this concept was introduced in \cite{VaananenCantorBendixson}.
Note that $X\subseteq{}^\kappa\kappa$ contains a perfect subset iff there is a continuous injection of ${}^\kappa 2$ into $X$ iff there is such a Borel injection; see, e.g.,
\cite[Proposition~2]{Friedman2014}.

Let $X$ be an analytic subset of the $\kappa$-Baire space ${}^\kappa\kappa$. We say that $R$ is a \emph{$\Fsigma$ (closed, Borel, etc.) relation on $X$} if,
for some $1\leq n<\omega$, 
$R$ is a 
$\Fsigma(\kappa)$ (closed, $\kappa$-Borel, etc.) subset of the product space ${}^n X$, where $X$ is endowed with the subspace topology induced by the $\kappa$-Baire space. 
A subset $Y$ of $X$ is \emph{$R$-independent} iff for all pairwise distinct $y_0,\dots, y_{n-1}\in Y$ we have $(y_0,\dots, y_{n-1})\notin R$.

Our main result, whose proof is detailed in Section~\ref{Fsigma rels section}, is as follows. 
\\[5 pt]
{\bf Theorem~\ref{fkappa orders}} {\it 
  Assume $\Iminus\kappa$ and either that $\Diamond_\kappa$ or that $\kappa$ is inaccessible. 
  \\[5 pt]
  Suppose $R$ is a $\Fsigma$ relation on an analytic subset of the $\kappa$-Baire space ${}^\kappa\kappa$. 
Then either all $R$-independent sets are of size $\leq\kappa$,  or there exists a perfect $R$-independent set.}
  \\[5 pt]
  \indent
  We note that in the case that $R$ is a \emph{closed binary relation} on an analytic subset of ${}^\kappa\kappa$, only the hypothesis $\Iminus\kappa$ is needed for the above dichotomy to hold; 
see Remark~\ref{closed rels}.

This dichotomy 
is the uncountable generalization of 
\cite[Corollary~2.13]{Kubis}. (There, the result is formulated in terms of homogeneous sets of $G_\delta$ colorings on analytic spaces; see Corollary~\ref{colorings} for this formulation in the uncountable case.) 
The special case of the result for binary relations on Polish spaces is also mentioned in \cite[Remark~1.14]{ShelahBorelSq}. 
  
The above dichotomy for uncountable $\kappa$ does not follow from ZFC alone; if there exists a weak $\kappa$-Kurepa tree, for example, then it can not hold for $\kappa$. This further implies that when $V=L$, the dichotomy fails for all uncountable $\kappa$, 
 and that the consistency of this dichotomy is at least that of an inaccessible cardinal. 
(See Remark~\ref{PSP} for details.)
By Theorem~\ref{fkappa orders}, we have that the consistency of a measurable cardinal implies the consistency of the above dichotomy; however
we do not yet know its exact consistency strength.

  A specific case of Theorem~\ref{fkappa orders} is that the ($\kappa$-)Silver dichotomy holds for $\Fsigma$ equivalence relations on the $\kappa$-Baire space under the assumption $\Iminus\kappa$ together with either $\Diamond_\kappa$ or the inaccessibility of $\kappa$.
Recently, a considerable effort has been made to investigate set theoretical conditions implying (the consistency of) the satisfaction or the failure of the Silver dichotomy for Borel equivalence relations on the generalized Baire space, see, e.g., \cite{Friedman2014,FriedKul} and \cite[Section~4.2]{FriedHyttKul}. It would be worth investigating whether our hypotheses imply this more general case as well.
\vspace{5 pt}

In Section~\ref{models section}, we use the results of the previous section to obtain 
model theoretic dichotomies motivated by 
the spectrum problem.

Suppose $\kappa$ is a cardinal and $\psi$ is a sentence in 
$\anform{\LL{\kappa^+}\kappa}$ (i.e., it is a second order sentence of the form $\exists \bar R\,\varphi(\bar R)$ where $\bar R$ is a set of $\leq \kappa$ many symbols disjoint from the original vocabulary and $\varphi(\bar R)$ is an $\LL{\kappa^+}\kappa$ sentence in the expanded language). 
One obtains interesting questions by considering, instead of the number of non-isomorphic $\kappa$-sized models of $\psi$, the possible sizes of sets of such models which are pairwise non-elementarily embeddable, as in for example \cite{BaldwinDiverseClasses,ShelahIE}. 
More generally, the role of elementary embeddings may be replaced by embeddings preserving (in the sense of (\ref{ee def}) in Definition~\ref{elem emb def}) ``nice'' sets of formulas, possibly of some extension of first order logic.

We consider the case when 
the ``nice'' sets of formulas to be preserved
are what we call \emph{fragments} of $\LL{\kappa^+}\kappa$ (see Definition~\ref{fragment def}).
Examples of fragments of $\LL{\kappa^+}\kappa$ include not only the set of all first order formulas and the set $\At$ 
of all atomic formulas and their negations (the maps preserving these sets of formulas are elementary embeddings and embeddings, respectively), but also the infinitary logics $\LL\lambda\mu$, where $\omega\leq\mu\leq\lambda\leq\kappa$, 
and $n$-variable fragments of these logics.
In the case of fragments
$F\subseteq \LL{\kappa^+}\omega$ 
and sentences $\psi\in F$, 
the set of models of $\psi$ 
together with the embeddings preserving $F$
forms an abstract elementary class, and 
the corresponding
version of the above question
has been 
studied in, e.g., \cite{ShelahAEC1}.
To the best knowledge of the first author, 
this question has not been studied yet in the case
fragments of $\LL{\kappa^+}\kappa$ which are not subsets of $\LL{\kappa^+}\omega$. 

Since we are dealing with models of size $\kappa$ up to 
elementary embeddability or, more generally, 
up to embeddability 
by maps
preserving fragments of $\LL{\kappa^+}\kappa$ (and isomorphisms are a special case of such embeddings), 
we may assume that all models have domain $\kappa$. Accordingly, let $\Mod\kappa\psi$ denote the set of models of $\psi$ with domain $\kappa$,
and denote by $\inj\kappa$ the set of injective functions in ${}^\kappa\kappa$.
Then any embedding between elements of $\Mod\kappa\psi$ (preserving a fragment of $\LL{\kappa^+}\kappa$) is an element of $\inj\kappa$. 
It is natural to ask what happens when, in the above questions, 
the role of $\inj\kappa$ is replaced by a certain subset $H$ of $\inj\kappa$, i.e., 
when $\Mod\kappa\psi$ is considered up to 
only the embeddings which are in $H$ (and preserve the given fragment of $\LL{\kappa^+}\kappa$).
Notice that when $H$ is a subgroup of $\sym\kappa$, the above question reduces to considering models up to isomorphisms in $H$. 

  More precisely, let $\struktura A$ and $\struktura B$ be models with domain $\kappa$. 
For a subset $H$ of $\inj\kappa$,
we say that \struktura A is \emph{$H$-elementarily embeddable} into $\struktura B$ iff there is some function $h$
in $H$ that embeds $\struktura A$ elementarily into $\struktura B$. 
(In the special case that $H$ is a subgroup of $\sym\kappa$, we say that the two structures are \emph{$H$-isomorphic}.) 
Analogously, we define 
the more general notion of \emph{$(F,H)$-elementary embeddability} for fragments $F$ of $\LL{\kappa^+}\kappa$
(see Definition~\ref{elem emb def}).
We are interested in the possible sizes of pairwise non $(F,H)$-elementarily embeddable subsets of $\Mod\kappa\psi$.
By introducing the set $H$ of ``allowed embeddings'' as an extra parameter, we may study explicitly the
role the topological properties of $H$ play in the above question.

A subset $C$ of a topological space is defined to be \emph{$\kappa$-compact} if any open cover of $C$ has a subcover of size $<\kappa$, and $C$ is \emph{$K_\kappa$} if it can be written as the union of at most $\kappa$ many $\kappa$-compact subsets. 
A topological space is $K_\kappa$ if it is a $K_\kappa$ subset of itself.

The next dichotomy theorem, which is the main result of Section~\ref{models section}, 
gives an answer to the above question when
$H$ is a $K_\kappa$ subset of the $\kappa$-Baire space, or, 
 for certain fragments,
when
$H$ is a $K_\kappa$ subset of the product space $({}^\kappa\kappa,\producttop)$, where $\producttop$ is the product topology on the set ${}^\kappa\kappa$ obtained by equipping $\kappa$ with the discrete topology. 
\\[5 pt]
{{\bf Theorem~\ref{elem emb}} 
{\it 
Assume the set theoretical hypotheses of Theorem~\ref{fkappa orders}. 
Suppose that $H\subseteq\inj\kappa$, $F$ is a fragment of $\LL{\kappa^+}\kappa$ and $\psi$ is a sentence of $\anform{\LL{\kappa^+}\kappa}$. Suppose that either 
\\[2 pt]1. $H$ is a $K_\kappa$ subset of the $\kappa$-Baire space, or
\\2. $H$ is a $K_\kappa$ subset of the product space 
$({}^\kappa\kappa,\producttop)$ 
and $F\subseteq\LL{\kappa^+}\omega$.
\\[2 pt]If there are at least $\kappa^+$ many pairwise non $(F,H)$-elementarily embeddable models  
in $\Mod\kappa\psi$,
then there are perfectly many such models.
}}
\\[5 pt]
\indent
Theorem~\ref{elem emb} 
can be seen as uncountable version of \cite[Theorems~5.8 and~5.9]{SagiSz}.
We note that these two cited theorems of \cite{SagiSz} also follow from \cite[Corollary~2.13]{Kubis} or from \cite[Remark~1.14]{ShelahBorelSq}.  

In order to prove Theorem~\ref{elem emb}, we have to first show in the beginning of Section~\ref{models section} that for any fragment $F$ 
and $K_\kappa$ subset $H$ of the $\kappa$-Baire space,
$\Mod\kappa\psi$ can be viewed as an analytic subset of the $\kappa$-Cantor space 
\emph{on which $(F,H)$-elementary embeddability is a $\Fsigma$ binary relation}
(and when $F\subseteq\LL{\kappa^+}\omega$, this holds even when $H$ is $K_\kappa$ only in the product topology~$\producttop$).
This is done by
considering 
a Borel refinement (the Borel refinement $t_F$ induced by~$F$) 
of the canonical topology used to study the deep connections between model theory and generalized descriptive set theory (see~\cite{MeklerVaananen} and, e.g.,~\cite{VaananenGamesTrees} and~\cite{FriedHyttKul}), 
and
generalizing to the uncountable case an argument in \cite{MorleyVC}. 
These arguments allow us to obtain Theorem~\ref{elem emb} as a special case of our more general dichotomy result, Theorem~\ref{fkappa orders}.
In 
Theorem~\ref{elem emb},
the word ``perfect'' may refer to the topology $t_{F'}$ induced by \emph{any} fragment $F'$ of $\LL{\kappa^+}\kappa$ (see Proposition~\ref{whichever fragment}). 

We remark that when $\kappa$ is a non-weakly compact cardinal, the assumption $\Iminus\kappa$ implies that there are no $K_\kappa$ subsets of 
the $\kappa$-Baire space other than those of size $\leq\kappa$.
However, $K_\kappa$ sets of size $>\kappa$ exist
in the case of the product space $({}^\kappa\kappa,\producttop)$, 
or in the case of the $\kappa$-Baire space when $\kappa$ is weakly compact. (See Remark~\ref{weakly compact}.)

One possible motivation for investigating 
the above questions
for $K_\kappa$ subsets $H$, even in the case of the $\kappa$-Baire space for $\kappa$ non-weakly compact, is
the following.
Let $F$ be a fixed fragment of $\LL{\kappa^+}\kappa$ and equip $\Mod\kappa\psi$  with the topology $t_F$ described above. 
Consider, for each $H\subseteq\inj\kappa$, the $(F,H)$-elementary embeddability relation
$R^F_H$
viewed as a subset of
$\Mod\kappa\psi\times\Mod\kappa\psi$. 
Specifically,
the relation $R^{F}_{\inj\kappa}$ 
(of embeddability by any map preserving $F$)
corresponds to the original question where the set of ``allowed'' embeddings 
``has not been restricted''.
Because the standard base of the space $(\Mod\kappa\psi, t_F)$ is of size $\kappa$, 
it is possible to construct
a subset (even a submonoid) $H$ of $\inj\kappa$ of size $\leq\kappa$ such that
$R^F_H$ is dense in $R^F_{\inj\kappa}$.
On the one hand, the density of
$R^{F}_{H}$ in $R^{F}_{\inj\kappa}$ 
may be interpreted, on an intuitive level, to mean that 
``the action of $H$ on $\Mod\kappa\psi$ is locally similar to the action of $\inj\kappa$''.
On the other hand, $|H|\leq\kappa$ and is therefore a $K_\kappa$ 
subset of the $\kappa$-Baire space, 
which implies that 
our model theoretic dichotomy result Theorem~\ref{elem emb} is applicable in this case as well.
\\[5 pt]
{\bf Notation.}
The notation we use is mostly standard. In particular, for sets $X,Y$ and ordinals $\lambda$, ${}^X Y$ denotes the set of functions from $X$ into $Y$, and ${[X]}^\lambda$ denotes the set of subsets of $X$ which are of size $\lambda$.
Furthermore, $\id X$ denotes the identity function on $X$. When $n\in\omega$, we denote elements of ${}^n X$ by $\bar x$, and $x_i$ denotes the $i^{\textrm{th}}$ coordinate of $\bar x$. When $f\in{}^X Y$, we use the notation $f(\bar x)$ for the element of ${}^n Y$ whose $i^{\textrm{th}}$ coordinate is $f(x_i)$.

Unless otherwise mentioned, we assume the set ${}^\kappa\kappa$ is equipped with the $\kappa$-Baire topology. However, we will sometimes consider the product topology 
on the set ${}^\kappa\kappa$, where $\kappa$ is equipped with the discrete topology; we denote this space by
$({}^\kappa\kappa,\producttop)$.

\section{Independent sets of $\Fsigma$ relations}
\label{Fsigma rels section}
In this section, we prove the main result of the paper, Theorem~\ref{fkappa orders}. Some corollaries of this theorem are stated at the end of the section.

Accordingly, we assume throughout this section that the hypothesis $\Iminus\kappa$ holds. We also assume that $R$ is a $\Fsigma$ relation on an analytic subset $X$ of the $\kappa$-Baire space and that $B\subseteq X$ is an $R$-independent set of size $\kappa^+$. Our goal is to show that under the additional hypothesis of 
$\Diamond_\kappa$ or when $\kappa$ is inaccessible, 
$X$ has a perfect $R$-independent subset.

We start by describing the basic idea of the proof in the special case when $X={}^\kappa\kappa$; the more general version will follow easily from this one.
We are going to construct functions
$(p_\xi\in{}^{<\kappa}\kappa:\xi\in {}^{<\kappa}2)$ 
in such a way that for all $\eta\subseteq\xi\in{}^{<\kappa} 2$ we have that $p_\eta\subseteq p_\xi$ and that $p_{\xi\conc 0}$ and $p_{\xi\conc 1}$ are incomparable; 
closing the set of these $p_\xi$'s under initial subsequences will give us a perfect tree $T$. 
In the case that $\Diamond_\kappa$ holds or when $\kappa$ is inaccessible, 
further conditions can be made on the $p_\xi$'s  to ensure that the 
$\kappa$-branches of $T$ 
form an $R$-independent set. 
In order to be able to construct all the $p_\xi$'s, 
we also make sure that the basic open subsets $N_{p_\xi}$ 
are ``big'' in the following sense. 
By the assumption \Iminus\kappa, we can fix a $\kappa^+$-complete normal 
ideal $\mathcal I$ on the set $B$ and a subset $K$ of $\mathcal I^+$ which is dense and in which every descending chain of length $<\kappa$ has a lower bound. 
We will guarantee the existence of sets $B_\xi\in K$ (for all $\xi\in{}^{<\kappa}2$) such that $B_\xi\subseteq N_{p_\xi}$.

Lemmas~\ref{disjoint} to~\ref{szetval2} below are needed to ensure that these sets and functions 
can be constructed.
When stating these lemmas, we assume that $B$, $\mathcal I$ and $K$ are fixed and satisfy the requirements described in the above paragraph.
Note again that our assumption $\Iminus\kappa$ implies $\kappa^{<\kappa}=\kappa$.

\begin{lemma}\label{disjoint}
  Suppose that $p\in{}^{<\kappa}\kappa$ and $B'\in \mathcal I^+$ are such that $B'\subseteq N_p$. Then there exist $p_0, p_1\supseteq p$ such that
  \begin{enumerate}[1.]
    \item 
      $N_{p_0}\cap N_{p_1}=\emptyset$ 
      (i.e., $p_0\not\subseteq p_1$ and $p_1\not\subseteq p_0$), and
\item $N_{p_i}\cap B'\in \mathcal I^+$ for $i=0,1$.
  \end{enumerate}
\end{lemma}
\begin{proof}
Assume, seeking a contradiction, that for $p$ and $B'$ as above, no such $p_0, p_1$ exist. For any $\alpha<\kappa$ such that $\dom p<\alpha$, we have, by $B'\subseteq N_p$, that
\[
  B'=\bigcup\{N_s\cap B':p\subset s\in{}^\alpha\kappa\}.
\]
Thus, because $\mathcal I$ is a $\kappa^+$-complete ideal, $B'\in\mathcal I^+$ and $|{ }^\alpha\kappa|\leq\kappa$, our assumption implies that
there exists exactly one $s_\alpha\in{}^\alpha 2$ extending $p$ such that 
\[N_{s_\alpha}\cap B'\in\mathcal I^+ \textrm{ and }B'-N_{s_\alpha}\in \mathcal I.\] 
Our assumption also implies that we must have $s_\beta\subset s_\alpha$
for all $\dom p<\beta<\alpha<\kappa$,
and therefore
$s=\bigcup_{\dom p<\alpha<\kappa}s_\alpha$ is an element of ${}^\kappa\kappa$. Then $\bigcap_{\dom p<\alpha<\kappa} N_{s_\alpha}=\{s\}$, 
 and so, 
 \[ 
    B'\subseteq\{s\}\cup\bigcup_{\dom p<\alpha<\kappa}(B'-N_{s_\alpha})\in\mathcal I,
  \]
by the $\kappa^+$-completeness and the non-principality of the ideal $\mathcal I$, contradicting the assumptions of the lemma.
\end{proof}
   In the next two lemmas, $0<n<\omega$ and $S$ denotes a closed subset of $\,{}^n({}^{\kappa}\kappa)$ such that $B$ is $S$-independent. 
 \begin{lemma}\label{szetval} 
   Suppose that $p_0,\dots,p_{n-1}\in {}^{<\kappa}\kappa$ and $B_0,\dots, B_{n-1}\in \mathcal I^+$ are such that 
   \[\textrm{$B_i\subseteq N_{p_i}$ and
     $p_i\not\subseteq p_j$
    for all $i,j<n$ with $i\neq j$.} 
  \]
  Then there exist $q_0,\dots,q_{n-1}\in{}^{<\kappa}\kappa$ and $B'_0,\dots,B'_{n-1}\in K$ for which
  \begin{enumerate}[1.]
    \item$(N_{q_0}\times\ldots\times N_{q_{n-1}})\cap S=\emptyset$, 
    \item $p_i\subseteq q_i$ for all $i<n$, and
    \item $B'_i\subseteq N_{q_i}\cap B_i$ for all $i<n$.
\end{enumerate}
\end{lemma}
\begin{proof}
  Take an arbitrary $\bar x=(x_0,\dots,x_{n-1})\in B_0\times\ldots\times B_{n-1}$. The $x_i$ are pairwise distinct elements of the $S$-independent set $B$, and therefore $\bar x\notin S$. 
 Thus, since $S$ is closed and $x_i\in B_i\subseteq N_{p_i}$, 
 there exists a 
 $\bar q(\bar x)=(q_0(\bar x),\dots, q_{n-1}(\bar x))\in{}^n({}^{<\kappa}\kappa)$,
  such that 
    {$p_i\subseteq q_i(\bar x)\subseteq x_i$, and }
  \[
  N_{q_0(\bar x)}\times\ldots\times N_{q_{n-1}(\bar x)}\subseteq{}^n({}^{\kappa}\kappa)- S.
  \]

  Fix $\bar y\in B_0\times\ldots\times B_{n-2}$. Since the ideal $\mathcal I$ is $\kappa^+$-complete and does not contain $B_{n-1}$, and also because there are $\kappa^{<\kappa}=\kappa$ many possibilities for the $\bar q_i(\bar x)$'s, there exists $\bar q(\bar y)\in {}^n({}^{<\kappa}\kappa)$ such that the set
  \[
    A(\bar y)=\{x\in B_{n-1}: \bar q(\bar y,x)=\bar q(\bar y)\}
    \textrm{ is in }
    \mathcal I^+. 
  \]
  Similarly, using induction on $k\leq n$, we can define, for all $\bar y\in {B_0\times\ldots\times B_{n-(k+1)}}$
  elements $\bar q(\bar y)$ of ${}^n({}^{<\kappa}\kappa)$ for which 
   \[
    A(\bar y)=\{x\in B_{n-k}: \bar q(\bar y,x)=\bar q(\bar y)\}
    \textrm{ is in }
    \mathcal I^+. 
  \]
  Finally, for $k=n$, we obtain $\bar q=(q_0,\dots,q_{n-1})\in{}^n({}^{<\kappa}\kappa)$ and 
  $A=\{x\in B_0: \bar q(x)=\bar q\}$.
  
  To see that $q_0,\dots, q_{n-1}$ will satisfy the requirements of the lemma, define an element $\bar y=(y_0,\dots,y_{n-1})\in B_0\times\ldots\times B_{n-1}$ such that $\bar q=\bar q(\bar y)$, as follows: let $y_0\in A$ be arbitrary, and if $1\leq i\leq n-1$ and $y_0,\dots, y_{i-1}$ have been defined, then let $y_i$ be an arbitrary element of 
  $A(y_0,\dots, y_{i-1})$.
  Then, using that $\bar q=\bar q(\bar y)$, we have $p_i\subseteq q_i$ and 
$N_{q_0}\times\ldots\times N_{q_{n-1}}\subseteq {}^n{({}^\kappa\kappa)} - S$.

Furthermore, for $0\leq i\leq n-1$, the set $A(y_0,\dots, y_{i-1})$ is a subset of $B_i$ by definition, and also of $N_{q_i}$, because of the following: for all $x\in A(y_0,\dots, y_{i-1})$, we can find $\bar z\in B_{i+1}\times\ldots\times B_{n-1}$  such that 
\[
\bar q(y_0,\dots, y_{i-1},x,\bar z)=\bar q(y_0,\dots, y_{i-1},x)=\bar q,
\]
implying that $q_i\subseteq x$.
Therefore, by $A(y_0,\dots, y_{i-1})\in\mathcal I^+$ and the density of $K$ in $\mathcal I^+$, there exists a $B'_i\in K$ such that $B'_i\subseteq B_i\cap N_{q_i}$. 
\end{proof}

The next lemma will aid us in proving the theorem for $\Fsigma$ relations in the case $\kappa$ is inaccessible. 
\begin{lemma}\label{szetval2}
  Suppose $\mu<\kappa$ is a cardinal, and $(p_\alpha\in{}^{<\kappa}\kappa: \alpha<\mu)$ and $(B_\alpha\in\mathcal I^+:\alpha<\mu)$ are such that 
  \[\textrm{$B_\alpha\subseteq N_{p_\alpha}$ and
  $p_\alpha\not\subseteq p_\beta$}
    \textrm{ for all $\alpha,\beta<\mu$ with $\alpha\neq\beta$.} 
  \]
  Then there exist, for all $\alpha<\mu$, elements $q_\alpha\in{}^{<\kappa}\kappa$ and $B'_\alpha\in K$ for which 
 \begin{enumerate}[1.]
    \item $(N_{q_{\alpha_0}}\times\ldots\times N_{q_{\alpha_{n-1}}})\cap S=\emptyset$ for all pairwise different $\alpha_0,\dots,\alpha_{n-1}<\mu$,
    \item $p_\alpha\subseteq q_\alpha$ for all $\alpha<\mu$, and
    \item $B'_\alpha\subseteq N_{q_\alpha}\cap B_\alpha$ for all $\alpha<\mu$.
  \end{enumerate}
\end{lemma}
\begin{proof} To prove this lemma, we basically iterate the application of the previous one as many times as necessary (being careful that the iteration goes through at limit stages); the details are below. To simplify notation, we assume $n=2$.
 
 Let $\big((\alpha_\sigma, \beta_\sigma):\sigma<\mu'\big)$ be an enumeration of $\{(\alpha,\beta):\alpha,\beta\in \mu,\,\alpha\neq\beta\}$,
 where $\mu'=\mu$ in the case that $\mu\geq\omega$, and $\mu'=\mu^2-\mu$ for $\mu<\omega$.
 By transfinite induction for $\sigma\leq\mu'$, we define
 \[
   p^\sigma_\alpha\in{}^{<\kappa}\kappa \textrm{ and } B^\sigma_\alpha\in K
   \textrm{ for all $\alpha<\mu$ simultaneously,}
\]
in such a way that the three properties below are true for all $\alpha<\mu$ and 
$\tau<\sigma\leq\mu'$. 
For simplicity, we use the notation $P^\sigma_\alpha=N_{p^\sigma_\alpha}$.
\begin{enumerate}[(i.)]
  \item $(P^{\tau+1}_{\alpha_\tau}\times P^{\tau+1}_{\beta_\tau})\cap S=\emptyset$.
  \item $p_\alpha\subseteq p^\tau_\alpha\subseteq p^\sigma_\alpha$ and 
    $B_\alpha\supseteq B^\tau_\alpha\supseteq B^\sigma_\alpha$,
  \item $B^\sigma_\alpha\subseteq P^\sigma_\alpha$.
\end{enumerate}
This is enough, because once these sequences have been defined, 
$q_\alpha=p^{\mu'}_\alpha$ and $B'_\alpha=B^{\mu'}_\alpha$ will satisfy the conclusions of the lemma.
To see that item~\textrm{1} of the conclusion of the lemma is satisfied, note that we have $N_{q_\alpha}=P^{\mu'}_\alpha\subseteq P^{\tau+1}_\alpha$ for all $\alpha<\mu, \tau<\mu'$, and use property (i.) above.

To build the sequences, we start from $p^0_\alpha=p_\alpha$ (and $P^0_\alpha=N_{p_\alpha}$) and some $B^0_\alpha\in K$ such that $B^0_\alpha\subseteq B_\alpha$. For limit ordinals $\sigma$, we define 
\[
  p^\sigma_\alpha=\bigcup_{\tau<\sigma}p^\tau_\alpha
  \textrm{ (and thus }
  P^\sigma_\alpha=\bigcap_{\tau<\sigma}P^\tau_\alpha).
\]
Then, by $B^\tau_\alpha\in K$ and the closure property of $K$, 
there exists $B^\sigma_\alpha\in K$ such that $B^\sigma_\alpha\subseteq\bigcap_{\tau<\sigma}B^\tau_\alpha$.
Therefore, using that $B^\tau_\alpha\subseteq P^\tau_\alpha$ for all $\tau<\sigma$ and the definition of $p^\sigma_\alpha$, 
we also have 
$B^\sigma_\alpha\subseteq P^\sigma_\alpha$.

In the case of a successor ordinal $\sigma+1$, we define the elements $p^{\sigma+1}_{\alpha_\sigma}, p^{\sigma+1}_{\beta_\sigma}$ of ${}^{<\kappa}\kappa$ and the sets $B^{\sigma+1}_{\alpha_\sigma}, B^{\sigma+1}_{\beta_\sigma}$ to be those obtained from 
an application of Lemma~\ref{szetval} to the functions $p^\sigma_{\alpha_\sigma}, p^\sigma_{\beta_\sigma}$ and the sets $B^\sigma_{\alpha_\sigma}, B^\sigma_{\beta_\sigma}$.
For $\alpha_\sigma\neq\gamma\neq\beta_\sigma$, we leave everything unchanged, i.e., we let $p^{\sigma+1}_\gamma=p^\sigma_\gamma$ and $B^{\sigma+1}_\gamma=B^\sigma_\gamma$. Then these definitions will satisfy all three required properties.
\end{proof}

We now state and prove the main result of our paper.
\begin{theorem}\label{fkappa orders} 
  Assume $\Iminus\kappa$ and either that $\Diamond_\kappa$ or that $\kappa$ is inaccessible. 
  \\[5 pt]  
  Suppose $R$ is a $\Fsigma$ relation on an analytic subset of the $\kappa$-Baire space ${}^\kappa\kappa$.
  Then either all $R$-independent sets are of size $\leq\kappa$,  or there exists a perfect $R$-independent set.
\end{theorem}
We note that, as was mentioned in the Introduction, $\Iminus\kappa$ implies $\Diamond_\kappa$ for successor cardinals $\kappa>\aleph_1$. Furthermore, when $R$ is a closed binary relation on $X$, the above dichotomy follows from $\Iminus\kappa$ alone; see Remark~\ref{closed rels} below.
\begin{proof} 
  As throughout the rest of this section, we assume \Iminus\kappa. 
  We start by proving the theorem in the case that $R$ is a $\Fsigma$ relation on the whole space ${}^\kappa\kappa$. 
  Suppose $R$ is $n$-ary (where $0<n<\omega$), and that
  $(R_\alpha:\alpha<\kappa)$ are closed subsets of ${}^n({}^\kappa\kappa)$ such that $R=\bigcup_{\alpha<\kappa}R_\alpha$. 
  Let $B$ be an independent set of $R$ of size $\kappa^+$.
  By the hypothesis $\Iminus\kappa$, there is a $\kappa^+$-complete normal 
  ideal $\mathcal I$ on $B$ and a subset $K$ of $\mathcal I^+$ which is dense and in which every descending chain of length $<\kappa$ has a lower bound. 

  First, assume $\Diamond_\kappa$. Then, the following combinatorial principle 
  also holds:
\begin{quotation}\noindent
  there exists a sequence of $n$-tuples
\[
\big((x^0_\alpha,\dots, x^{n-1}_\alpha)\in{}^n({}^\alpha 2):\,\alpha<\kappa\big)
\]
such that
$x^i_\alpha\neq x^j_\alpha$ for all $i<j<n,\,\alpha<\kappa$, and 
for all $(x_0,\dots,x_{n-1})\in{}^n({}^\kappa 2)$ 
with $x_i\neq x_j$ for all $i<j<n$ we have that
the set
$\{\alpha\in\kappa: x^i_\alpha= {x_i\restr{\alpha}}\textrm{ for all }i<n\}$ is cofinal in~$\kappa$. 
\end{quotation}
We note that whenever $\kappa$ is not inaccessible, the above combinatorial principle is equivalent to $\Diamond_\kappa$, by the Proposition in \cite{MatetCorrigendumDiamond}. (See also \cite{DevlinVariationsDiamond} and \cite{PiolunowiczDiamond}, from which this equivalence follows when $\kappa=\omega_1$ and when $\kappa$ is a successor cardinal, respectively.)

We define sets $(B_\xi\in K: \xi\in {}^{<\kappa}2)$ and functions $(p_\xi\in{}^{<\kappa}\kappa: \xi\in {}^{<\kappa}2)$ 
  in such a way that the items below are satisfied for all $\xi,\eta\in {}^{<\kappa}2$ and $\alpha<\kappa$. We use the notation $P_\xi=N_{p_\xi}$.
  \begin{enumerate}[1.]
    \item $B_\xi\subseteq P_\xi$;
    \item if $\eta\subseteq \xi$, then $B_\eta\supseteq B_\xi$ and $p_\eta\subseteq p_\xi$;
    \item if $\xi\neq\eta\in {}^\alpha 2$, then 
    $P_\xi\cap P_\eta=\emptyset$ (i.e.,
    $p_\xi$ and $p_\eta$ are incomparable);
    \item for all $\gamma<\alpha$ we have 
      \[(P_{x^0_\alpha}\times\ldots\times P_{x^{n-1}_\alpha})\cap R_\gamma=\emptyset
      \]
      (i.e., for any $\bar t\in{}^n({}^\kappa\kappa)$ such that $t_i\supseteq p_{x^i_\alpha}$ for all $i<n$, we have $\bar t\not\in R_\gamma$).
\end{enumerate}
By the second and third conditions, the map 
$t:{}^\kappa 2\longrightarrow{}^\kappa\kappa;
\,x\mapsto\bigcup_{\alpha<\kappa}p_{x\restriction\alpha}$ 
is a continuous injection whose image $\textit{Im}(t)$ is a perfect subset of ${}^\kappa\kappa$.
Furthermore, item~4 guarantees that 
$\textit{Im}(t)$
will be $R$-independent:
if $x_0,\dots, x_{n-1}\in {}^\kappa 2$ are pairwise distinct and $\gamma<\kappa$ is arbitrary, 
then there exists a $\gamma<\alpha<\kappa$ such that 
$x^i_\alpha={x_i\restr\alpha}$ for all $i<n$, 
and therefore, by 4, we have that $\big(t(x_0),\dots, t(x_{n-1})\big)\notin R_\gamma$.
  
  Thus, it is enough to see that sets $B_\xi$ and functions $p_\xi$ satisfying the conditions above can indeed be constructed. 
To guarantee the first three conditions, we will use Lemma~\ref{disjoint} and the density of $K$ at successor stages, and we will use the closure property of $K\subseteq\mathcal{I}^+$ at limit stages. 
Then, the fourth property can be guaranteed 
using Lemma~\ref{szetval}. The details are below.

Let $p_\emptyset=\emptyset$ and let $B_\emptyset\in K$ be arbitrary. Now, fix $\alpha<\kappa$ and suppose that $p_\eta$ and $B_\eta$ have been defined for all $\eta\in{}^{<\alpha}2$ so that conditions 1 to 4 hold. We first construct elements $p'_\xi\in{}^{<\kappa}\kappa$ and $B'_\xi\in K$, for all $\xi\in{}^\alpha 2$ which satisfy the first three required conditions, i.e., for all $\xi,\eta\in{}^\alpha 2$ with $\xi\neq\eta$, 
\begin{multline*}
   p'_\xi\not\subseteq p'_\eta 
   \textrm{ and } B'_\xi\subseteq N_{p'_\xi},
\\
\textrm{and for all $\beta<\alpha$ we have
  $p'_\xi\supseteq p_{\xi_{|\beta}}$ and $B'_\xi\subseteq B_{\xi_{|\beta}}$.}
\end{multline*}
If $\alpha=\beta+1$, then, for each $\eta\in{}^\beta 2$, apply Lemma~\ref{disjoint} for $p_\eta$ and $B_\eta$ to obtain $p'_{\eta\conc 0}$ and $p'_{\eta\conc 1}$. By the density of $K$ in $\mathcal I^+$, we can find 
$B'_{\eta\conc 0},B'_{\eta\conc 1}\in K$ such that $B'_{\eta\conc i}\subseteq N_{p'_{\eta\conc i}}\cap B_\eta$ for $i=0,1$.  
If $\alpha$ is a limit ordinal, then for all $\xi\in{}^\alpha 2$, let 
$p'_\xi=\bigcup_{\beta<\alpha} p_{\xi\restriction\beta}$, and, using the closure property of $K$, choose sets
$B'_\xi\in K$ such that $B'_\xi\subseteq \bigcap_{\beta<\alpha} B_{\xi\restriction\beta}$. 
We have $B'_\xi\subseteq N_{p'_\xi}$ because $B_{\xi\restriction\beta}\subseteq P_{\xi\restriction\beta}$ for all $\beta<\alpha$ and by the definition of $p'_\xi$. 

After $p'_\xi$ and $B'_\xi$ have been constructed for all $\xi\in{}^\alpha 2$, we apply Lemma~\ref{szetval} for the closed subset $\bigcup_{\gamma<\alpha} R_\gamma$ of ${}^n({}^\kappa\kappa)$ and for $p'_{x^0_\alpha},\dots,p'_{x^{n-1}_\alpha}$ and $B'_{x^0_\alpha},\dots,B'_{x^{n-1}_\alpha}$ in order to obtain the elements $p_{x^0_\alpha},\dots,p_{x^{n-1}_\alpha}$ and $B_{x^0_\alpha},\dots,B_{x^{n-1}_\alpha}$. 
In the case that $\xi\notin\{x^0_\alpha,\dots,x^{n-1}_\alpha\}$, we leave everything unchanged, i.e., we let $p_\xi=p'_\xi$ and $B_\xi=B'_\xi$. This choice of the $p_{\xi}$'s and $B_{\xi}$'s will guarantee that all four required properties are satisfied.

\vspace{5 pt}
Now, assume that $\kappa$ is inaccessible. We modify the previous argument in every step of the construction, by requiring, instead of condition~4, that the following property is satisfied for all $\alpha<\kappa$:
\begin{enumerate}
  \item[$4'.$] for all pairwise distinct $\xi_0,\dots,\xi_{n-1}\in{}^{\alpha} 2$  and $\gamma<\alpha$, we have 
    \[(P_{\xi_0}\times\ldots\times P_{\xi_{n-1}})\cap R_{\gamma}=\emptyset.\]
\end{enumerate}
This property will ensure that the perfect set $\textit{Im}(t)$ is an independent set of $R_\gamma$ for each $\gamma<\kappa$ and therefore of $R$. 
Condition $4'$ can be guaranteed in the $\alpha^{\textrm{th}}$ step of the construction by applying Lemma~\ref{szetval2} for the closed subset $\bigcup_{\gamma<\alpha} R_\gamma$ of ${}^n({}^\kappa\kappa)$ and for the collections $(p'_\xi:\xi\in{}^\alpha 2)$ and $(B'_\xi:\xi\in{}^\alpha 2)$ of the previous argument; note that $|{}^\alpha 2|<\kappa$ by the inaccessibility of $\kappa$. 
This completes the proof of the theorem in the case that $R$ is a $\Fsigma$ relation on the whole space ${}^\kappa\kappa$.

\vspace{5 pt}

Now, assume that $R$ is an $n$-ary $\Fsigma$ relation on a closed subset $X$ of the $\kappa$-Baire space ${}^\kappa\kappa$. Then $R'=R\cup({}^n({}^\kappa\kappa)-{}^n X)$ is an $n$-ary $\Fsigma$ relation on~${}^\kappa\kappa$
(note that the assumption $\kappa^{<\kappa}=\kappa$ implies that all open subsets of~${}^\kappa\kappa$ are also $\Fsigma$).
Furthermore, the $R'$-independent sets of size at least $n$ are subsets of $X$ and are therefore also $R$-independent. Hence, this case follows by applying the previous case for $R'$. 

Finally, suppose that $X$ is an arbitrary analytic subset of ${}^\kappa\kappa$ and $R$ is an $n$-ary $\Fsigma$ relation on $X$.
We can assume that $\Delta_n=\{\bar x\in{}^n X: x_i=x_j \textrm{ for some }i<j<n\}$ is a subset of $R$ (we can take $R\cup \Delta_n$ instead of $R$ if necessary, because $\Delta_n$ is a closed relation on $X$ and because, by definition, a set $Y\subseteq X$ is $(R\cup \Delta_n)$-independent iff it is $R$-independent).

Because $X$ is analytic, there exists a closed subset $X''$ of ${}^\kappa\kappa$ and a continuous function $f:{}^\kappa\kappa\longrightarrow {}^\kappa\kappa$ such that $X=f[X'']$. Denoting by $f(\bar x)$ the vector $(f(x_0),\dots, f(x_{n-1}))$ (for $\bar x\in {}^n X$), we let
  \[R''=\{
    \bar x\in {}^n X'': f(\bar x)\in R\}.
  \]
  That is, $R''$ is the inverse image of $R$ under the continuous function 
  ${{}^n X''\longrightarrow {}^n X;}$
  ${\bar x\mapsto f(\bar x)},
  $
  and is therefore a $\Fsigma$ $n$-ary relation on $X''$.

 If $B$ is an $R$-independent set of size $\kappa^+$, then any   
  $B''\subseteq X''$ such that $f[B'']=B$ and $f\restr B''$ is injective 
 is an $R''$-independent set (by the definition of $R''$and the injectivity of $f\restr B''$), and $B''$  has cardinality $\kappa^+$. Therefore, by the previous case, 
  $R''$ has a perfect independent set, i.e., there is a continuous injection $t$ of ${}^\kappa 2$ into $X''$ whose image is $R''$-independent.
 Then, by the definition of $R''$ and because $\Delta_n\subseteq R$, the function $f\comp t$ is a continuous injection  of ${}^\kappa 2$ into $X$ whose image is $R$-independent, and thus, $X$ has a perfect $R$-independent subset.
\end{proof}

\begin{remark}\label{closed rels}
  When $R$ is a \emph{closed binary} relation on an analytic subset $X$ of the $\kappa$-Baire space, the dichotomy in Theorem~\ref{fkappa orders} already follows from the hypothesis $\Iminus\kappa$, 
  even when neither $\Diamond_\kappa$ nor the inaccessibility of $\kappa$ is assumed.
  This can be seen by a slight modification of the proof of Theorem~\ref{fkappa orders}, which we give below.

  We first describe the proof for a closed subset $X$ of the $\kappa$-Baire space; the more general case of analytic subsets follows just as in the proof of Theorem~\ref{fkappa orders}. 
  (However, the argument for obtaining the statement for closed sets $X$ from the $X={}^\kappa\kappa$ case no longer works, since $({}^\kappa\kappa\times{}^\kappa\kappa)- (X\times X)$ is not necessarily closed when $X$ is closed.)

  Suppose that $R$ is a closed subset of $X\times X$ and $B\subseteq X$ is an $R$-independent set of size $\kappa^+$, and let $\mathcal I$ and $K$ be as in the proof of Theorem~\ref{fkappa orders}. As before, we define, for all $\xi\in {}^{<\kappa}2$, functions $p_\xi\in {}^{<\kappa}\kappa$ and sets $B_\xi\in K$ satisfying certain conditions. 
  We require  items~1 to~3 of the proof of Theorem~\ref{fkappa orders}; these conditions can be guaranteed exactly as in that proof. 
  For $R$-independence, 
  it is enough this time to ensure, using Lemma~\ref{szetval} (for $n=2$), that 
  \begin{enumerate}
    \item[$4''.$] $(P_{\xi\conc0}\times P_{\xi\conc{1}})\cap R=\emptyset$ for all $\xi\in{}^{<\kappa}2$
\end{enumerate}
(where $P_\nu=N_{p_\nu}$). 
That the $\kappa$-branches defined by two distinct $x,y\in{}^\kappa 2$ will not be $R$-related follows from applying item~$4''$ to the node $\xi\in{}^{<\kappa}2$ at which $x$ and $y$ split.
The sufficiency of item~$4''$ is the reason
the extra assumption of either $\Diamond_\kappa$ or the inaccessibility of $\kappa$
is unnecessary.

We have to add a fifth condition to guarantee that the perfect set defined by the $p_\xi$'s is a subset of $X$. 
Since $X$ is closed and $\kappa^{<\kappa}=\kappa$ holds, 
there exist open sets $X_\alpha$ such that $X=\bigcap_{\alpha<\kappa} X_\alpha$. We require that
\begin{enumerate}
  \item[$5.$] $P_\xi\subseteq X_{\alpha}$ for all $\alpha<\kappa$ and $\xi\in{}^\alpha 2$.
\end{enumerate}
Suppose $\xi\in{}^\alpha 2$ and a $p'_\xi$ and $B'_\xi$ has been defined satisfying items~1-3 and~$4''$. Then 
$B'_\xi\subseteq B\subseteq X\subseteq X_\alpha$, and 
so, (using item~1) we have $B'_\xi\subseteq
N_{p'_\xi}\cap X_\alpha$. 
Therefore, using the facts that $B'_\xi\in\mathcal I^+$, that $\mathcal I$ is $\kappa^+$-complete and that $N_{p'_\xi}\cap X_\alpha$ is the union of at most $\kappa$ many basic open sets, we can find an extension  $p_\xi\supseteq p'_\xi$ such that $N_{p_\xi}\subseteq X_\alpha$ and $N_{p_\xi}\cap B'_\xi\in\mathcal I^+$. By the density of $K$ in $\mathcal I^+$, there exists $B_\xi\in K$ which is a subset of $N_{p_\xi}\cap B'_\xi$, and thus, $p_\xi$ and $B_\xi$ satisfy all~5 required conditions.
This completes the proof of the case when $X$ is a closed subset of the $\kappa$-Baire space.

To obtain from this case the dichotomy for closed binary relations $R$ on \emph{arbitrary analytic subsets} $X$ of ${}^\kappa\kappa$, we employ exactly the same argument (for $n=2$) as the one that was used at end of the proof of Theorem~\ref{fkappa orders} to show the corresponding implication for $\Fsigma$ relations. Note that,
using the notation at the end of the proof of Theorem~\ref{fkappa orders}, $R''$ is now a closed binary relation on a closed subset $X''$ of ${}^\kappa\kappa$, because it is the inverse image of the closed subset $R$ of $X\times X$ under a continuous function.
\end{remark}
\begin{remark}\label{PSP} 
  The special case of the dichotomy in Theorem~\ref{fkappa orders} for the empty binary relation $R=\emptyset$ is that the \emph{$\kappa$-perfect set property} holds for closed (and even analytic) subsets of the $\kappa$-Baire space, i.e., any closed (analytic) subset is either of size $\leq\kappa$ or contains a perfect subset. (In fact, by a generalization of \cite[Theorem~1]{VaananenCantorBendixson}, or as a special case of Remark~\ref{closed rels}, the $\kappa$-perfect set property for closed sets follows already from the assumption $\Iminus\kappa$.) 
  
  This implies that our dichotomy does not hold for $\kappa$ if there exists a \emph{weak $\kappa$-Kurepa tree}, i.e., a tree of height $\kappa$ with at least $\kappa^+$ many $\kappa$-branches (of length $\kappa$) whose $\alpha^{\textrm{th}}$ level has size $\leq|\alpha|$ for stationarily many $\alpha\in\kappa$. If $T$ is a weak $\kappa$-Kurepa tree, then the set $[T]$ of $\kappa$-branches of $T$ is a closed subset of the $\kappa$-Baire space of size~$\geq\kappa^+$ that does not contain a perfect subset; see \cite[Section~4]{FriedHyttKul} 
  and \cite{Friedman2014}. 
  The idea of using Kurepa trees to obtain counterexamples to the $\aleph_1$-perfect set property had already appeared in \cite{VaananenCantorBendixson} and \cite{MeklerVaananen}. 
  
  We therefore have the following two facts. Firstly, in $V=L$, our dichotomy does not hold for any uncountable regular $\kappa$, by \cite[Lemma~4]{Friedman2014}. 
  Secondly, our dichotomy for $\kappa$ implies, by a result of Robert Solovay \cite[Sections~3 and~4]{JechTrees}, that $\kappa^+$ is an inaccessible cardinal in $L$. 
  Note that by Theorem~\ref{fkappa orders}, the consistency of our dichotomy follows from the existence of a measurable $\lambda>\kappa$; however, we do not yet know its exact consistency strength.\\[10 pt]
\end{remark}

As mentioned in the Introduction, Theorem~\ref{fkappa orders} can also be seen as the uncountable version of a result \cite[Corollary~2.13]{Kubis} about homogeneous subsets of $G_\delta$ colorings on analytic spaces. 
A \emph{coloring} of a set $X$ is any subset $C$ of ${[X]}^n$ for some $1<n\in\omega$, and a set $B\subseteq X$ is \emph{$C$-homogeneous} if ${[B]}^n\subseteq C$.
Notice that ${[X]}^n$ can be identified with an open subset of ${}^n X$, namely 
$\{
  \bar x\in {}^n X: x_i\neq x_j \textrm{ for all } i<j<n
\}.$
Thus, when $X$ is an analytic subset of the $\kappa$-Baire space, we may speak about $\Gdelta$ colorings on $X$. 
The ``complement'' of such colorings can be considered as $\Fsigma$ relations on $X$, and furthermore, the concept of homogeneity translates to that of independence. Hence, Theorem~\ref{fkappa orders} yields the following.
\begin{corollary}\label{colorings} Assume $\Iminus\kappa$ and either that $\Diamond_\kappa$ or that $\kappa$ is inaccessible.
  \\[5 pt]
  Suppose that $C$ is a $\Gdelta$ coloring on an analytic subset of ${}^\kappa\kappa$. Then either all $C$-homogeneous sets are of size $\leq\kappa$, or there is a perfect $C$-homogeneous set.
\end{corollary}

In the remainder of the section, we state some corollaries of the $n=2$ case of Theorem~\ref{fkappa orders}.
\begin{corollary}\label{Kkappa actions}
Assume $\Iminus\kappa$ and either that $\Diamond_\kappa$ or that $\kappa$ is inaccessible.
  Then the $\kappa$-Silver dichotomy holds for $\Fsigma$ equivalence relations on analytic subsets of ${}^\kappa\kappa$. 
\end{corollary}
We use the following notation in the next corollary. For $H$ a topological space, $X$ any set and 
$S\subseteq H\times X\times X$,  
  let $R_S$ be the projection of $S$ onto $X\times X$, i.e., 
  \[R_S=\{(x,y): (h,x,y)\in S\textrm{ for some }h\in H\}.\]
  Specifically, for the action $a$ of a group $H$ on $X$, $R_a$ is the orbit equivalence relation.
  
  Recall from the introduction that a topological space is $\kappa$-compact
  if any of its open covers has a subcover of size $<\kappa$, and is $K_\kappa$, if it can be written as the union of at most $\kappa$ many $\kappa$-compact subsets.
A topological group is $K_\kappa$ if it is $K_\kappa$ as a topological space. 
The first item of the next corollary is stated 
  because 
  we will need to use this more general form later.
  
\begin{corollary}\label{fkappa orders2}
Assume $\Iminus\kappa$ and either that $\Diamond_\kappa$ or that $\kappa$ is inaccessible.
  Let $X$ be an analytic subset of the $\kappa$-Baire space (equipped with the subspace topology) and let $H$ be an arbitrary $K_\kappa$ topological space. 
  \begin{enumerate}[1.] 
\item If 
$S\subseteq H\times X\times X$ is closed,
then either all $R_S$-independent sets have size $\leq\kappa$ or there is a perfect $R_S$-independent set. 
\item If $H$ is a group that acts continuously on $X$, then there are either $\leq\kappa$ many or perfectly many orbits.
\end{enumerate}
\end{corollary}
\begin{proof} A generalization of a standard argument from the countable case \cite[Exercise~3.4.2]{SuGao} shows that if $H$ is a $\kappa$-compact topological space, then $R_S$ is a closed subset of $X\times X$. In more detail, let $(x,y)\in X\times X-R_S$ be arbitrary. Because $S$ is closed, we can choose for all $h\in H$ open sets $U_h\subseteq H$ and $V_h\subseteq X\times X$ such that 
\[(h,x,y)\in U_h\times V_h\subseteq H\times X\times X - S.\] 
By the $\kappa$-compactness of $H$, there exists a set $I\in{[H]}^{<\kappa}$ such that $H=\bigcup_{h\in I}U_h$. Then $V=\bigcap_{h\in I} V_h$ is an open subset of $X\times X$ such that $(x,y)\in V\subseteq X\times X - R_S$. 

Thus $R_S$ is indeed closed in the case $H$ is $\kappa$-compact,  
implying that if $H$ is $K_\kappa$, then $R_S$ is a $\Fsigma$ binary relation on $X$. An application of Theorem~\ref{fkappa orders} completes the proof of \textrm{1}, of which the second item is a special case. 
\end{proof}

\begin{remark}\label{weakly compact}
In the rest of the paper, we will be 
considering
the special cases of
$K_\kappa$ topological spaces obtained by 
endowing $K_\kappa$ subsets of the $\kappa$-Baire space with the subspace topology (induced by the $\kappa$-Baire topology). 
We also consider
$K_\kappa$ subsets of the product space $({}^\kappa\kappa,\producttop)$ 
equipped
with the subspace topology 
induced by $\producttop$.
(Recall from the introduction that $\producttop$ is the product topology on the set ${}^\kappa\kappa$ obtained by endowing $\kappa$ with the discrete topology.)
We therefore give characterizations of, or some examples for, such sets below.

Notice that $K_\kappa$ subsets of the $\kappa$-Baire space are always $K_\kappa$ subsets of $({}^\kappa\kappa,\producttop)$
as well, due to the fact that the $\kappa$-Baire topology on the set ${}^\kappa\kappa$ is finer than the product topology $\producttop$.

When $\kappa$ is a weakly compact cardinal, the converse also holds; in this case, the following are equivalent for any $H\subseteq{}^\kappa\kappa$. 
\begin{enumerate}[1.]
  \item $H$ is a $K_\kappa$ subset of the $\kappa$-Baire space. 
  \item $H$ is a $K_\kappa$ subset of the product space $({}^\kappa\kappa,\producttop)$. 
  \item $H$ is a $\Fsigma$ subset of the $\kappa$-Baire space which is \emph{eventually bounded} 
    (this means, by definition, that
    there exists an $x\in{}^\kappa\kappa$ such that for all $h\in H$, we have $h\leq^\ast x$, 
    i.e., 
    there exists an $\alpha<\kappa$ such that $h(\beta)\leq x(\beta)$ for all $\alpha\leq\beta<\kappa$). 
\end{enumerate}
    The equivalence of the first and third items follows from \cite[Lemma~2.6]{LuckeMottoRosSchlicht}, and, as just noted, the second item is implied by the first one. 
  
    The third item can be obtained from the second using the fact that a $\kappa$-compact subset $C$ of $({}^\kappa\kappa,\producttop)$ is 
  closed in the $\kappa$-Baire topology 
  and is bounded 
  (i.e., there exists an $x\in{}^\kappa\kappa$ such that $c(\beta)\leq x(\beta)$ for all $c\in C$ and $\beta<\kappa$). 
   The above fact can be proven by modifying standard arguments from the countable case (see \cite[Exercise~4.11]{KechrisClassicalDST}). 
   In more detail, suppose that $C\subseteq{}^\kappa\kappa$ is $\kappa$-compact in the product topology~$\producttop$. 
   If we take any $\alpha<\kappa$
   (and we denote by $N_{\{(\alpha,\gamma)\}}$ the set $\{y\in{}^\kappa\kappa:y(\alpha)=\gamma\}\in\producttop$ for all $\gamma<\kappa$), the family
   \[\{N_{\{(\alpha,\gamma)\}}:\gamma<\kappa
     \text{ and }
   C\cap N_{\{(\alpha,\gamma)\}}\neq\emptyset\}\]
   is a disjoint $\producttop$-open cover of $C$, and must therefore be of size $<\kappa$.
   Thus, we can define $x\in{}^\kappa\kappa$ by letting 
   $x(\alpha)=\sup\{\gamma<\kappa:C\cap N_{\{(\alpha,\gamma)\}}\neq\emptyset\}$ for all $\alpha<\kappa$, 
 and $x$ witnesses that $C$ is bounded.

   To see that $C$ is closed in the $\kappa$-Baire topology,
   let $z\in{}^\kappa\kappa-C$. We can choose, for all $y\in C$, disjoint neighborhoods $U_y, V_y\in\producttop$ of $z$ and $y$ respectively. 
 By the $\kappa$-compactness of $C$, the $\producttop$-open cover $\{V_y:y\in C\}$ of $C$ 
   can be refined to a subcover $\{V_y:y\in I\}$ of size $<\kappa$. Then the intersection $U=\bigcap_{y\in I}U_y$ is disjoint from $C$, contains $z$, and is open in the $\kappa$-Baire topology (because the $\kappa$-Baire topology is finer than $\producttop$ and 
is closed under intersections of size $<\kappa$).
   \\[5 pt]
\indent
Now, consider the case when $\kappa^{<\kappa}=\kappa$ is not weakly compact. If $\Iminus\kappa$ holds,  
the $K_\kappa$ subsets of the \emph{$\kappa$-Baire space}
  are exactly those of size at most $\kappa$. This is because by \cite[Corollary~2.8]{LuckeMottoRosSchlicht}, the above equivalence holds when $\kappa^{<\kappa}=\kappa$ is not weakly compact and the closed subsets of the $\kappa$-Baire space have the $\kappa$-perfect set property (this latter requirement 
  is implied by $\Iminus\kappa$;
  see Remark~\ref{PSP}). 
  Examples of subsets $H$ of the set ${}^\kappa\kappa$ which are $\kappa$-compact 
  subsets of the \emph{product space $({}^\kappa\kappa,\producttop)$}
  but are not $K_\kappa$ 
  subsets of the $\kappa$-Baire space, even when $\Iminus\kappa$ is not assumed,
  include by \cite[Lemma~2.2]{LuckeMottoRosSchlicht} (and Tychonoff's theorem) the set $H={}^\kappa 2$ and, 
  more generally, sets of the form $H=\prod_{\alpha<\kappa}I_\alpha$, where $I_\alpha\in{[\kappa]}^{<\kappa}$ and $I_\alpha$ is finite except for~$<\kappa$ many $\alpha<\kappa$. 
  (See \cite{Lipparini2013} and the references therein for when this last assumption can be weakened.)
\end{remark} 
\begin{remark}\label{weakly compact2}
  In light of the previous remark, especially the fact that $K_\kappa$ subsets of the $\kappa$-Baire space are of size $\leq\kappa$ when $\kappa$ is not weakly compact and $\Iminus\kappa$ holds, we mention the following observation.
  
  A model of ZFC in which $\kappa$ is weakly compact (in fact, supercompact) and $\Iminus\kappa$ holds can be obtained, starting out from a situation in which $\kappa$ is supercompact and there exists a measurable $\lambda>\kappa$, in the following way.
Before L\'evy-collapsing $\lambda$ to $\kappa^+$,
one first applies the Laver preparation~\cite{Laver1978} to make the supercompactness of $\kappa$ indestructible by any $<\kappa$-directed closed forcing. 
If we would also like to have $2^\kappa>\kappa^+$ together with $\Iminus\kappa$ for a supercompact $\kappa$, we force after the Laver-preparation with the product of the L\'evy-collapse and $\textit{Add}(\kappa, \mu)$ for some $\mu>\lambda$.
We would like to thank Menachem Magidor for suggesting the arguments found in this remark. 
\end{remark}

\section{
  Elementary embeddability on models of size $\kappa$
}\label{models section}
In this section, we use our previous results to obtain dichotomy properties for the set of models with domain $\kappa$ of some $\anform{\LL{\kappa^+}\kappa}$-sentence, when considered up to isomorphism, embeddability or elementary embeddability by elements of some $K_\kappa$ subset $H$ of the $\kappa$-Baire space. 
More generally, we will be interested in these models 
up to maps (in $H$) that preserve
certain subsets of $\LL{\kappa^+}\kappa$-formulas, which we will call \emph{fragments} (see Definition~\ref{fragment def} below). 
In
the case of some fragments, it will be enough to assume that $H$ is a $K_\kappa$ subset of the product space 
$({}^\kappa\kappa,\producttop)$; 
(recall from the introduction that $\producttop$ is the product topology on the set ${}^\kappa\kappa$, where $\kappa$ is given the discrete topology).
\\[5 pt]
{\bf Notation.}
We will use the following notation in this section. As usual, $\sym\kappa$ denotes the permutation group of $\kappa$, and we write $\inj\kappa$ for the monoid of all injective functions of ${}^\kappa\kappa$.

The symbol $L$ denotes a fixed first order language which contains only relation symbols and is of size at most $\kappa$.
However, the arguments below also work in the case of languages which have 
infinitary relations of arity $<\kappa$.
We assume the language $L$ has $\kappa$ many variables, the sequence of which is denoted by $(v_i: i\in\kappa)$.
The symbols $\struktura A$, $\struktura B$, etc. are used to denote $L$-structures whose domains are $A$, $B$, etc. The set of all $L$-structures with domain $\kappa$ is denoted by \Mod \kappa L. 
Given a structure $\struktura A\in \Mod \kappa L$, we identify $\struktura A$-valuations with elements of ${}^\kappa\kappa$.

As usual, $\LL\lambda\mu$ 
(where $\omega\leq\mu\leq\lambda\leq\kappa^+$ and $\mu\leq\kappa$) 
denotes the infinitary language which allows conjunctions and disjunctions of $<\lambda$ many formulas and quantification over $<\mu$ many variables.
See, e.g.,
Definitions 1.1.2 and 1.1.3 of \cite{DickmannInfinitaryLanguages} 
for the precise definition of {$\LL{\lambda}\mu$-formulas} 
or, alternatively, Definitions 9.12 and 9.13 of \cite{VaananenModelsGames}.
In particular, note that 
by definition, an $\LL{\lambda}\mu$-formula contains $<\mu$ many free variables,
from $\{v_i:i\in\kappa\}$.
The concept of the \emph{subformulas} of a formula $\varphi\in\LL{\kappa^+}\kappa$ is defined by induction on the complexity of $\varphi$ as usual 
(see, e.g., Definition 1.3.1 of \cite{DickmannInfinitaryLanguages} or \cite[p.~234]{VaananenModelsGames}.
Note that if $\varphi$ is obtained as $\varphi=\bigwedge \Phi$, then any $\phi\in\Phi$ is defined to be a subformula of $\varphi$, but $\bigwedge\Phi'$, where $\Phi'\subset\Phi$, is not a subformula). 

For $\varphi\in \LL{\kappa^+}\kappa$ and $h\in{}^\kappa\kappa$, we denote by $s_h \varphi$ the formula obtained from $\varphi$ by simultaneously substituting, for all $i\in\kappa$, the variable $v_{h(i)}$ for the variable $v_i$. 
We say that a set $F$ of $\LL{\kappa^+}\kappa$-formulas is \emph{closed under substitution} if for any $\varphi\in F$ and $h\in {}^\kappa\kappa$ we have $s_h \varphi\in F$. 
Note that since 
${\kappa}^{<\kappa}=\kappa$, closing a nonempty set of formulas of $\LL{\kappa^+}\kappa$ of size $\leq\kappa$ under substitution leads to a set of formulas of size $\kappa$.

\begin{definition}
\label{fragment def}
We define a subset $F$ of the infinitary language $\LL{\kappa^+}\kappa$ 
to be a \emph{fragment of $\LL{\kappa^+}\kappa$} 
if $|F|=\kappa$,
\begin{enumerate}[1.]
  \item $F$ contains all atomic formulas,
  \item $F$ is closed under negation and taking subformulas, and
  \item $F$ is closed under substitution of variables.
\end{enumerate}
\end{definition}
\begin{definition}
\label{elem emb def}
Suppose $F$ is a fragment of $\LL{\kappa^+}\kappa$, 
$H\subseteq\inj\kappa$,
and 
$\struktura A,\struktura B\in\Mod\kappa L$.
We say that $\struktura A$ is \emph{$(F,H)$-elementarily embeddable} into \struktura B 
iff there is a map $h\in H$ such that 
\begin{equation}\label{ee def}
\struktura A\models\varphi[a] \textrm{ iff } 
\struktura B\models\varphi[h\comp a]
\textrm{ for all $\varphi\in F$ and valutations $a\in{}^\kappa\kappa$.}
\end{equation}
\end{definition}
Note that
any function 
$h\in{}^\kappa\kappa$ satisfying (\ref{ee def}), where $F$ is any fragment, 
must be 
an embedding of $\struktura A$ into $\struktura B$ (by item~1 
of Definition~\ref{fragment def}), and in particular, 
we must have $h\in\inj\kappa$. 
This is the reason that we have chosen to define the above concept only for subsets $H$ of $\inj\kappa$ (instead of arbitrary subsets of ${}^\kappa\kappa$).
We remark that if $H$ is a \emph{submonoid} of $\inj\kappa$, then the binary relation of $(F,H)$-elementary embeddability on $\Mod\kappa L$
is a partial order (but there is no reason for this to hold when $H$ is an arbitrary subset of $\inj\kappa$).
\begin{example}
Many interesting notions are special cases 
of $(F,H)$-elementary embeddability, including 
\begin{itemize*}
  \item
$H$-embeddability (when $F$ is the set $\At$ of atomic formulas and their negations), 
\item
$H$-elementary embeddability (when $F=\LL\omega\omega$), 
and
\item $H$-isomorphism, when $H$ is a subgroup of $\sym\kappa$ (and $F$ is any fragment. Two models in $\Mod\kappa L$ are defined to be \emph{$H$-isomorphic} iff there exists a $h\in H$ which is an isomorphism between them.) Note that 
  $H$-isomorphism is an equivalence relation on $\Mod\kappa L$ when $H$ is a subgroup of $\sym\kappa$.
\end{itemize*}
  Other maps that may be of interest are 
  \begin{itemize*}
    \item
  those obtained when $F=\LL\lambda\mu$ (where $\omega\leq\mu\leq\lambda\leq\kappa$), 
or 
\item
those preserving the $n$-variable fragment $L^n_{\lambda\omega}$ of this logic.
(By definition, the $n$ variable fragment of $\LL\lambda\mu$, or equivalently of $\LL\lambda\omega$, consists of those formulas which use only the variables $v_0,\dots,v_{n-1}$. In this case, the corresponding fragment $F$ is the set of those $\LL\lambda\omega$-formulas which contain at most $n$ (arbitrary) variables from $\{v_i:i\in\kappa\}$. It is this fragment $F$ that we will denote by $L^n_{\lambda\omega}$.) 
\end{itemize*}
\end{example}

A fragment $F\subseteq\LL{\kappa^+}\kappa$ induces a topology $t_F$ on the set \Mod\kappa L in a natural way.  
To a formula $\varphi$ and a valuation $a\in{}^{\kappa}\kappa$, 
we correlate the basic clopen set
$\mod\kappa\varphi a=\{\struktura A\in\Mod\kappa L: \struktura A\models\varphi[a]\}.
$
The topology $t_F$ on \Mod\kappa L is obtained by taking arbitrary unions of intersections of $<\kappa$ many sets from the collection 
\[b_F=\{\mod\kappa\varphi a :\varphi\in F, a\in{}^{\kappa}\kappa\},\]
and we denote by $\Modf F$  
the topological space $(\Mod\kappa L,t_F)$.
The canonical topological space 
used to study the connections between model theory and the generalized Baire space
is $\Modf\At$, and it
is homeomorphic to the Cantor space ${}^\kappa 2$; see \cite{MeklerVaananen,VaananenGamesTrees,FriedHyttKul}. 
An advantage of working with $t_{F}$ instead of $t_{\At}$ is that $(F,H)$-elementary embeddability induces a ``$t_{F}$-continuous action'' of $H$ on $\Mod\kappa\psi$ 
(in the more general sense of item~1 of Corollary~\ref{fkappa orders2}; see the proof of Theorem~\ref{elem emb} below). 
The above fact is needed in order for us to be able to use the results of the previous section. 

Firstly, we show that for an arbitrary fragment $F$, the space $\Modf F$ is homeomorphic to a $\Gdelta$ subset $X_F$ of ${}^\kappa 2$. Our proof is basically a generalization from the countable case of a proof in \cite{MorleyVC}. 
To start, note that a bijection between $\kappa$ and $F$ allows us to define the generalized Cantor topology on ${}^F 2$. In fact, since $\kappa$ is regular (by $\kappa^{<\kappa}=\kappa$), 
another basis for this topology is 
$
\{N_p: p\in{}^\Phi 2 \textrm{ for some } \Phi\in{[F]}^{<\kappa}\},
$
where $N_p=\{x\in {}^F 2 : p\subseteq x\}$,
and therefore
this topology does not depend on the chosen bijection.

For a fragment $F$ of $\LL{\kappa^+}\kappa$, define an injection $i_F:\Mod\kappa L\longrightarrow {}^F 2$ as follows: if $\struktura A\in\Mod\kappa L$, then let $i_F(\struktura A)\in {}^F 2$ be such that 
\[i_F(\struktura A)(\varphi)=1 \textrm{ iff } \struktura A\models\varphi[\id\kappa]
\]
for all $\varphi\in F$.
The function $i_F$ is in fact injective, because if $\struktura A, \struktura B$ are different structures in $\Mod\kappa L$, then there exists a formula $\varphi$ in $\At$ (the set of all atomic formulas and their negations) such that 
$\struktura A\models\varphi[\id\kappa]$ and $\struktura B\not\models\varphi[\id\kappa]$, 
  i.e., 
 \[\textrm{ 
$
i_F(\struktura A)(\varphi)=1$, and $i_F(\struktura B)(\varphi)=0$. 
}\]
This implies, using the fact that $\At\subseteq F$ by Definition~\ref{fragment def}, that $i_F(\struktura A)\neq i_F(\struktura B)$.

We denote by $X_F$ the image of the injection $i_F$.
\begin{proposition}\label{models1}
  If $F$ is a fragment of $\LL{\kappa^+}\kappa$, then $X_F$ is a $\Gdelta$ subset of the Cantor space ${}^F 2$, and $i_F$ is a homeomorphism from $\Modf F$ onto its image $X_F$, where $X_F$ is equipped with the subspace topology.
\end{proposition}
\begin{proof}
  Because $F$ is closed under substitution, the collection $b_F$ from which the topology $t_F$ is obtained is actually equal to 
  $\{\mod\kappa\varphi{\id\kappa}:\varphi\in F\}.$ 
  Using this fact, it 
  is not hard to see that the injection $i_F$ is a homeomorphism between $\Modf F$ and its image $X_F$.
  
  To see that $X_F\subseteq {}^F 2$ is $\Gdelta$, we define the following subsets of ${}^F 2$. For any $h\in{}^\kappa\kappa$, we denote by $\mathrm{supp}(h)$ the set of those $\alpha\in\kappa$ for which $h(\alpha)\neq\alpha$.
 \begin{align*}
  X_0=\{x\in{}^F 2&: x(\psi)=1\textrm{ iff } x(\neg\psi)=0\textrm{ for all $\psi\in F$}\};\\
  X_1=\{x\in{}^F 2&:\textrm{if $\psi\in F$ and $\psi=\bigwedge\Phi$ for some $\Phi\in{[F]}^{\leq\kappa}$,}\\ &\textrm{ then }x(\psi)=1\textrm{ iff } \textrm{for all $\varphi\in\Phi$ we have }x(\varphi)=1\};\\
  X_2=\{x\in{}^F 2&:\textrm{if $\psi\in F$ and $\psi=\exists(v_\beta:\beta\in I)\varphi$ where $\varphi\in F$ and $ I\in{[\kappa]}^{<\kappa}$,}\\ &\textrm{ then }x(\psi)=1\textrm{ iff } x(s_h\varphi)=1
  \\&\hspace{4cm}
  \textrm{for some $h\in{}^{\kappa}\kappa$ } \textrm{such that }\mathrm{supp}(h)\subseteq I
\};\\
  X_3=\{x\in{}^F 2&: \textrm{for all $i,j\in\kappa$ we have }x(v_i=v_j)=1\textrm{ iff } i=j\},
\end{align*}
and we let $X$ denote their intersection.
We claim that the $X_i$'s are all $\Gdelta$ subsets of ${}^F 2$ and show this in detail for $X_1$.
For $\psi=\bigwedge \Phi\in F$, 
the set
\[
  X^1_{1,\psi}=\{x\in{}^F 2:x(\psi)=1\textrm{ and } x(\varphi)=1\textrm{ for all $\varphi\in\Phi$}\}
  =\bigcap_{\varphi\in\Phi}N_{\{(\psi,1),(\varphi,1)\}}
\]
is $\Gdelta$, while the set
\[
  X^0_{1,\psi}=\{x\in{}^F 2:x(\psi)=0\textrm{ and } x(\varphi)=0\textrm{ for some $\varphi\in\Phi$}\}
  =\bigcup_{\varphi\in\Phi}N_{\{(\psi,0),(\varphi,0)\}}
\]
is open. Therefore 
$X_1=
\bigcap 
\{X^0_{1,\psi}\cup X^1_{1,\psi}:\psi=\bigwedge\Phi\in F\}$ is also a $\Gdelta$ subset of the Cantor space ${}^F 2$. 
That $X_0$, $X_2$ and $X_3$ are also $\Gdelta$ can be seen similarly; in the case of $X_2$, one has to use the fact that $|\{s_h\varphi:h\in{}^\kappa\kappa\}|\leq\kappa$ holds, because $\varphi$ has $<\kappa$ variables and ${\kappa}^{<\kappa}=\kappa$.

Therefore the intersection $X$ of the $X_i$'s is also $\Gdelta$, and it remains to see that $X_F=X$.
It is straightforward to show that for any $\struktura A\in\Mod\kappa L$, we have $i_F(\struktura A)\in X$, and so $X_F\subseteq X$. 
For the other direction, first observe that if $x,y\in X$ and $x\restr\At=y\restr\At$, then $x=y$ also holds, by an easy induction on the complexity of formulas. 
Suppose $x\in X$ is arbitrary. We wish to define a model $\struktura A\in\Mod\kappa L$
such that $x=i_F(\struktura A)$; by the above observation, it is enough to require that ${x\restr\At}={{i_F(\struktura A)}\restr\At}$. 
Clearly, the $L$-model $\struktura A$ whose domain is $\kappa$ and whose relations are defined by letting,
for each $n$-ary relation symbol $R$ of $L$ and $\alpha_1,\ldots,\alpha_n\in\kappa$,
\[(\alpha_1,\ldots,\alpha_n)\in R^{\struktura A}\textrm{ iff } x(R(v_{\alpha_1},\ldots,v_{\alpha_n}))=1,
\]
satisfies these requirements.
\end{proof}

Let $\psi$ be an arbitrary sentence in $\anform{\LL{\kappa^+}\kappa}$ (recall that this means $\psi$ is a second order sentence of the form $\exists \bar R\,\varphi(\bar R)$ where $\bar R$ is a set of $\leq\kappa$ many symbols disjoint from $L$ and $\varphi(\bar R)$ is an $\LL{\kappa^+}\kappa$-sentence in the language expanded by~$\bar R$). 
We let $\Mod\kappa \psi$ be the set of elements of $\Mod\kappa L$ which are models of $\psi$. If $F$ is a fragment of $\LL{\kappa^+}\kappa$, we denote by $\Modt\psi F$ the corresponding subspace of the topological space $\Modf F$, or in other words, $\Modt\psi{F}$ is obtained by endowing $\Mod\kappa \psi$ with the topology $t_F$. We furthermore define 
\[X^\psi_F=\{i_F(\struktura A):\struktura A\in\Mod\kappa \psi\},\]
the set of elements of $X_F$ corresponding to models of $\psi$. 
\begin{corollary}\label{models2}
  If $\psi$ is a sentence of $\anform{\LL{\kappa^+}\kappa}$ and $F$ is a fragment of $\LL{\kappa^+}\kappa$, 
  then $X^\psi_{F}$ is an analytic subset of ${}^F 2$. Furthermore, 
  the map $i_F\restr\Modt\psi{F}$ is a homeomorphism from $\Modt\psi{F}$ onto its image $X^\psi_{F}$ equipped with the subspace topology.
\end{corollary}
\begin{proof} 
As we have seen at the beginning of the previous proof, it is enough to show that $X^\psi_F$ is analytic.
  First, in the case when $\psi\in F$ (and therefore is an $\LL{\kappa^+}\kappa$-sentence), we have 
$X^\psi_F=
 X_F\cap N_{\{(\psi,1)\}}.
$
Thus, $X^\psi_F$ is a $\Gdelta$ subset of ${}^F 2$ by Proposition~\ref{models1}.

Now, in general, suppose that $\psi$ is the sentence $\exists \bar R\varphi(\bar R)$.
 Let $F'$ be the fragment 
 generated (in the expanded language) by $F\cup \{\varphi(\bar R)\}\cup \bar R$.
Then we have that $X^\psi_F$ is the image of the $\Gdelta$ set $X^\varphi_{F'}$ under the continuous map ${}^{F'}2\longrightarrow {}^F 2,$ $x\mapsto x\restr F$ and is therefore analytic. (Equivalently, $\Modt\psi F$ is the image of $\Modt\varphi{F'}$ under the continuous map defined by taking the $L$-reducts of models for 
the expanded language.)
\end{proof}

We now turn to the ``number'' of pairwise non $(F,H)$-elementarily embeddable models of an arbitrary sentence $\psi\in\anform{\LL{\kappa^+}\kappa}$, where $H\subseteq\inj\kappa$ and $F$ is a fragment of $\LL{\kappa^+}\kappa$. 
We will say that $\psi$ has \emph{perfectly many non $(F,H)$-elementarily embeddable models} if 
the binary relation of $(F,H)$-elementary embeddability on $\Mod\kappa\psi$ has a $t_{F}$-perfect independent set. (Note that we may speak about $t_{F}$-perfect sets since $\Modt\psi F$ is homeomorphic to a subset of the $\kappa$-Baire space.) 

We remark that, as Proposition~\ref{whichever fragment} below shows, 
the choice of the fragment that generates the topology on $\Mod\kappa\psi$ is actually irrelevant in the above definition. That is,
given any fragment $F'$ of $\LL{\kappa^+}\kappa$,
the sentence $\psi$ has perfectly many non $(F,H)$-elementarily embeddable models iff the binary relation of $(F,H)$-elementary embeddability on $\Mod\kappa\psi$ has a $t_{F'}$-perfect independent set.

Below, by a $t_F$-Borel subset of $\Mod\kappa\psi$, we mean a $\kappa$-Borel subset of $\Modt\psi F$,
and a map $f:X\longrightarrow \Mod\kappa\psi$ (where $X$ is a topological space) is $t_F$-Borel iff the inverse images of $t_F$-Borel subsets of $\Mod\kappa\psi$ are $\kappa$-Borel subsets of $X$.
\begin{proposition}\label{Borel maps}
  Let $F$ and $F'$ be arbitrary fragments of $\LL{\kappa^+}\kappa$.
  \begin{enumerate}[1.]
    \item A subset of $\Mod\kappa\psi$ is $t_F$-Borel iff it is $t_{F'}$-Borel.
    \item A map $f:{}^\kappa 2\longrightarrow \Mod\kappa\psi$ is $t_F$-Borel iff it is $t_{F'}$-Borel.
  \end{enumerate}
\end{proposition}
\begin{proof}
  An easy induction shows that for any $\varphi\in\LL{\kappa^+}\kappa$ (and therefore for any $\varphi\in F'$) and valuation $a\in{}^\kappa\kappa$, the set $\mod\kappa\varphi a$ is $t_F$-Borel. Consequently, all $t_{F'}$-Borel sets are $t_F$-Borel sets as well. 
  By symmetry, we have item~1, of which item~2 is a direct consequence. 
\end{proof}

\begin{proposition}\label{whichever fragment}
  Let $F$ and $F'$ be arbitrary fragments of $\LL{\kappa^+}\kappa$, and suppose $R$ is a binary relation on $\Mod\kappa\psi$. Then $R$ has a $t_F$-perfect independent set iff it has a $t_{F'}$-perfect independent set.
\end{proposition}
  \begin{proof}
    The binary
    relation $R$ has a $t_F$-perfect independent set iff there is a {$t_F$-Borel} (instead of ``$t_F$-continuous'') injection of ${}^\kappa 2$ into $\Mod\kappa\psi$ such that 
    \begin{equation*}\tag{$\ast$}\label{reduction}
  \textrm{for all }x\neq y\in {}^\kappa 2, \textrm{ we have } (r(x),r(y))\notin R;
\end{equation*}
see for example    
\cite[Proposition~2]{Friedman2014}
for a proof. (A map satisfying (\ref{reduction}) is called a \emph{reduction} of $\id{{}^\kappa 2}$ to $R$.)
Therefore, using item~2 of Proposition~\ref{Borel maps} above, we obtain the conclusion of this proposition immediately.
\end{proof}

We now state and prove the main theorem of this section.
\begin{theorem}\label{elem emb}
  Assume $\Iminus\kappa$ and either that $\Diamond_\kappa$ holds or that $\kappa$ is inaccessible.
  Suppose that $H\subseteq\inj\kappa$, $F$ is a fragment of $\LL{\kappa^+}\kappa$ and $\psi$ is a sentence of $\anform{\LL{\kappa^+}\kappa}$. Suppose that either 
  \begin{enumerate}[1.]
    \item $H$ is a $K_\kappa$ subset of the $\kappa$-Baire space, or
    \item $H$ is a $K_\kappa$ subset of the product space $({}^\kappa\kappa,\producttop)$ 
      and $F\subseteq\LL{\kappa^+}\omega$.
  \end{enumerate}
If there are at least $\kappa^+$ many pairwise non $(F,H)$-elementarily embeddable models  
in $\Mod\kappa\psi$,
then there are perfectly many such models.
\end{theorem}
Note that since the $\kappa$-Baire topology is finer than the product topology $\producttop$, being a $K_\kappa$ subset of the $\kappa$-Baire space implies being a $K_\kappa$ subset of the product space $({}^\kappa\kappa,\producttop)$. 
This implication is strict if and only if $\kappa$ is not weakly compact. 
Moreover, if $\kappa$ is not weakly compact and $\Iminus\kappa$ holds, then the $K_\kappa$ subsets of the $\kappa$-Baire space are exactly the ones of size $\leq\kappa$ (see Remark~\ref{weakly compact}).

We note that, (as mentioned in the Introduction), it is consistent to also have $2^\kappa>\kappa^+$ together with the set theoretical hypothesis of the above theorem, relative to the consistency of a measurable $\lambda>\kappa$. Furthermore, starting from a stronger situation, we can also assume that $\kappa$ is a weakly compact cardinal such that $\Iminus\kappa$ and $2^\kappa>\kappa^+$ (see Remark~\ref{weakly compact2}).
\begin{proof}
We start with the proof of item~1. Assume, as usual, that $H$ is given the subspace topology induced by the $\kappa$-Baire space and 
 $X^\psi_{F}$ is equipped with the subspace topology induced by the generalized Cantor space ${}^F 2$. 
  
  Define the subset $S$ of $H\times X^\psi_{F}\times X^\psi_{F}$ by letting, for any $h\in H$ and $\struktura A,\struktura B\in\Mod\kappa\psi$,
\begin{align*}
  (h,i_F(\struktura B),i_F(\struktura A))\in S \textrm{ iff }&\textrm{$h$ witnesses that $\struktura A$ is}
\\
&\textrm{$(F,H)$-elementarily embeddable 
into $\struktura B$}.
\end{align*}
    Then, since $F$ is closed under substitution and negation, $(h,i_F(\struktura B),i_F(\struktura A))\in S$ iff for all formulas $\varphi$ in $F$,
    $\struktura A\models\varphi[\id\kappa]$ implies that $\struktura B\models\varphi[{h\comp\id\kappa}]$, or equivalently that $\struktura B\models s_h\varphi[\id\kappa].$
Therefore
\[S=\{(h,y,x)\in H\times X^\psi_{F}\times X^\psi_{F}: \textrm{ for all $\varphi\in F$}, \textrm{ } x(\varphi)=1 \textrm{ implies }y(s_h\varphi)=1\}.\]
 
  \begin{claim}\label{S is closed}
$S$ is a closed subset of $H\times X^\psi_{F}\times X^\psi_{F}$.
\end{claim}
\noindent\textit{Proof of Claim~\ref{S is closed}.}
  We prove that the complement $U=H\times X^\psi_{F}\times X^\psi_{F} - S$ is open. Suppose that $(h,y,x)\in U$, or in other words, there exists $\varphi\in F$ such that $x(\varphi)=1$ and $y(s_h\varphi)=0$. 
  Then, since the set $\Delta(\varphi)$ of free variables of $\varphi$ is of size $<\kappa$, the set $N_1=N_{h\restriction{\Delta(\varphi)}}\cap H$ is an open subset of $H$. 
  Furthermore, $h'\in N_1$ implies that for all $x\in X^\psi_F$, we have $x(s_{h'}\varphi)=x(s_h\varphi)$. 
  Thus, denoting by $N_2$ and $N_3$ the open subsets of $X^\psi_{F}$ determined by the conditions $z(\varphi)=1$ and $z(s_h\varphi)=0$, respectively, we obtain an open neighborhood $N_1\times N_2\times N_3$ of $(h,y,x)$ which is also a subset of $U$.
  This completes the proof of Claim~\ref{S is closed}.
  \\[5 pt] 
  Clearly, the projection $R_S$ of $S$ onto $X^\psi_{F}\times X^\psi_{F}$ is the relation corresponding to $(F,H)$-elementary embeddability on $\Mod\kappa\psi$ (i.e., 
$(i_F(\struktura B),i_F(\struktura A))\in R_S$ iff \struktura A is $F$-elementarily embeddable into \struktura B by $H$).
By Corollary~\ref{models2}, $X^\psi_F$ is an analytic subset of the generalized Cantor space ${}^{F}2$, and $H$ is a $K_\kappa$ topological space by the assumption of item~1.
Thus, by
Corollary~\ref{fkappa orders2}, we have the required conclusion.
\\[5pt]
\indent
To see item~\textrm{2}, we equip $H$ with the subspace topology induced by 
the \emph{product topology $\producttop$} on the set ${}^\kappa\kappa$. 
As before, $X^\psi_F$ is given the subspace topology induced by the generalized Cantor space ${}^F 2$, and the set $S$ is defined as above. 
Then, using the assumption $F\subseteq\LL{\kappa^+}\omega$ 
of item~2, one can show that
\[
  \textrm{$S$ is a closed subset of the space $H\times X^\psi_F\times X^\psi_F$}.
\] 
  This can be seen by 
  using
  the argument in the proof of Claim~\ref{S is closed} 
  and taking note of the fact that, since
  the set of free variables of any $\varphi\in F$ is finite 
  by the 
  assumption $F\subseteq\LL{\kappa^+}\omega$, 
  the set denoted by $N_1$ in the proof of Claim~\ref{S is closed} is an open subset of $H$ even when the topology on $H$ is inherited from  
  the product space $({}^\kappa\kappa,\producttop)$.

Furthermore, $H$ is a $K_\kappa$ topological space by the first assumption of item~2, 
and $X^\psi_F$ is an analytic subset of the generalized Cantor space ${}^F 2$ by Corollary~\ref{models2}. Therefore Corollary~\ref{fkappa orders2} can again be applied to obtain the required conclusion.
\end{proof}
\begin{remark}
  Suppose that $\Phi$ is an arbitrary $\kappa$-sized subset of $\LL{\kappa^+}\kappa$-formulas which is closed under substitution.
  For a subset $H$ of ${}^\kappa\kappa$, 
  consider the 
  models in $\Mod\kappa\psi$ up to 
  maps $h\in H$ which preserve the formulas in $\Phi$ (i.e., maps $h:\struktura A \longrightarrow \struktura B$ such that for all $\varphi\in\Phi$ and valuations $a\in{}^\kappa A$, if $\struktura A\models\varphi[a]$ then $\struktura B\models\varphi[h\comp a]$. Note that in this case, such maps $h$ need not be injective). 
  Using the topology $t_F$, where $F$ is the fragment generated by $\Phi$,
 the proof of Theorem~\ref{elem emb} can be generalized 
 to yield an analogous statement about the ``number of models'' up to such maps. 
  This version seems to cover all natural generalizations of Theorem~\ref{elem emb}.

  Specifically, when $\Phi$ is the set of those atomic formulas which do not contain the $=$ symbol, a map $h$ preserves $\Phi$ iff it is a homomorphism. Therefore, when $H$ is a $K_\kappa$ subset of the product space $({}^\kappa\kappa,\producttop)$,
  we have that the following holds under the assumptions of Theorem~\ref{fkappa orders}: \emph{ if there are at least $\kappa^+$ many models in $\Mod\kappa\psi$ such that no $h\in H$ is a homomorphism from one into another, then there are perfectly many such models.}
\end{remark}

We conclude this section by mentioning some interesting special cases of Theorem~\ref{elem emb}.

\begin{corollary}\label{elem emb2} 
  Assume that $\Iminus\kappa$ and either $\Diamond_\kappa$ or that $\kappa$ is inaccessible. Let $H\subseteq\inj\kappa$ and let $\psi$ be a sentence of $\anform{\LL{\kappa^+}\kappa}$.
  \begin{enumerate}[1.]
    \item Suppose $H$ is 
      a $K_\kappa$ subset of the product space $({}^\kappa\kappa,\producttop)$.
      If there are at least $\kappa^+$ many pairwise non $H$-elementarily embeddable models in $\Mod\kappa\psi$, then there is a perfect set of such models.
    \item The above also holds for $H$-embeddability, as well as $(F,H)$-elementary embeddability when $F$ is either $\LL\lambda\omega$ or $L^n_{\lambda\omega}$, and $n<\omega\leq\lambda\leq\kappa$.
    \item Suppose $H$ is a $K_\kappa$ subset of the $\kappa$-Baire space. Then the same statement holds for $(\LL\lambda\mu,H)$-elementary embeddability, where $\omega<\mu\leq\lambda\leq\kappa$.
    \item Now, suppose that $H$ is a subgroup of $\sym\kappa$ which is $K_\kappa$ again in the product topology $\producttop$. If there are $\kappa^+$ many pairwise non $H$-isomorphic models in $\Mod\kappa\psi$, then there is a perfect set of such models.
  \end{enumerate}
\end{corollary}

\subsection*{Acknowledgements}
The research of the first author was partially supported by the Central European University Budapest Foundation (CEUBPF)
and by 
the Hungarian National Research, Development and Innovation Office grant no. 113017.
The authors would like to thank Menachem Magidor for suggesting arguments found in the Introduction and in Remark~\ref{weakly compact2}.
The first author would like to thank Lajos Soukup, Tapani Hyttinen, Vadim Kulikov, Philipp Schlicht and Boban Veli\u ckovi\'c for valuable discussions and G\'abor S\'agi for reading and making useful comments on this manuscript.
\bibliographystyle{abbrv}
\bibliography{references}
 \end{document}